\documentclass[a4paper,12pt]{amsart}
\usepackage{amsmath, amssymb}
\usepackage[dvipdfmx]{hyperref}
\hypersetup{colorlinks=true,linkcolor=blue,citecolor=blue}

\newtheorem{thm}[equation]{Theorem}
\newtheorem{prop}[equation]{Proposition}
\newtheorem{lem}[equation]{Lemma}
\newtheorem{cor}[equation]{Corollary}
\newtheorem{clm}[equation]{Claim}

\theoremstyle{definition}
\newtheorem{defn}[equation]{Definition}

\theoremstyle{remark}
\newtheorem{rem}[equation]{Remark}
\newtheorem{ex}[equation]{Example}

\newcommand{\N}{\ensuremath{\mathbb{N}}}

\newcommand{\R}{\ensuremath{\mathbb{R}}}
\newcommand{\tX}{{\tilde{X}}}

\newcommand{\tZ}{\tilde{Z}}
\newcommand{\mux}{\mu_{X}}
\newcommand{\muy}{\mu_{Y}}
\newcommand{\muz}{\mu_{Z}}

\newcommand{\seqn}[1]{ \{#1\}_{n=1}^{\infty} }
\DeclareMathOperator{\diam}{diam}
\DeclareMathOperator{\pt}{pt}

\newcommand{\niN}{{n\in\N}}
\newcommand{\nti}{{n\to\infty}}

\newcommand{\ep}{\varepsilon}
\renewcommand{\phi}{\varphi}

\newcommand{\cA}{\mathcal{A}}
\newcommand{\cK}{\mathcal{K}}
\newcommand{\cL}{\mathcal{L}}
\newcommand{\cM}{\mathcal{M}}
\newcommand{\cO}{\mathcal{O}}
\newcommand{\cP}{\mathcal{P}}
\newcommand{\cQ}{\mathcal{Q}}
\newcommand{\cX}{\mathcal{X}}
\newcommand{\cY}{\mathcal{Y}}
\newcommand{\proh}{d_{P}}
\DeclareMathOperator{\pr}{pr}

\title[Boundedness of precompact sets of metric measure spaces]{Boundedness of precompact sets of metric measure spaces}
\author[D. Kazukawa]{Daisuke Kazukawa}
\address{Department of Mathematics, Osaka University, Toyonaka, Osaka 560-0043, JAPAN}
\email{d-kazukawa@cr.math.sci.osaka-u.ac.jp}
\author[T. Yokota]{Takumi Yokota}
\address{Mathematical Institute, Tohoku University, Sendai 980-8578, JAPAN}
\email{takumiy@tohoku.ac.jp}


\begin{document}
\renewcommand{\thefootnote}{\fnsymbol{footnote}}
\footnote[0]{{\it Date}: May 1, 2021.}
\maketitle
\begin{abstract}
We give a detailed proof to Gromov's statement that precompact sets of metric measure spaces are bounded with respect to the box distance and the Lipschitz order.
\end{abstract}

\section{Introduction}

A \emph{metric measure space}, or an \emph{mm-space} for short, is a triple $X=(X, d_X, \mux)$,
where $\mux$ is a Borel probability measure on a complete separable metric space $(X, d_X)$.
Among other things, Gromov~\cite{G} introduced a distance function $\Box$, called the \emph{box distance},
and a partial order $\prec$, called the \emph{Lipschitz order}, on the set $\cX$ of all (isomorphism classes of) mm-spaces.
Our main reference is Shioya's book~\cite{S}.

For mm-spaces $X$ and $Y$,
we say that \emph{$X$ dominates $Y$} and write $Y\prec X$ if there exists a $1$-Lipschitz map $f:X\to Y$ with $f_* \mux=\muy$,
e.g.~\cite[Definition~2.10]{S}.

The purpose of this note is to give a detailed proof of the following statement, which was only sketched in \cite{G}.
\begin{thm}[Gromov~{\cite[3$\frac{1}{2}$.15.(f)]{G}}]\label{thm;main}
If $\cY\subset\cX$ is a precompact set in $(\cX, \Box)$,
then there exists an mm-space $X$ with $Y\prec X$ for any $Y\in\cY$.
\end{thm}

In Theorem~\ref{thm;main},
the product of spaces in $\cY$ does the job if $\cY$ is a finite set,
cf.~\cite[Proposition~4.57]{S},
but we do not know such a simple construction of $X$ in general.

\begin{ex}\label{ex;precpt}
Let $K\in\R$, $N\in\N$, and $D>0$.
The set $\cX^D_{\le N}$ of finite mm-spaces with cardinality $\le N$ and diameter $\le D$ is a typical example of $\Box$-compact sets,
e.g.~\cite[Theorem~4.27]{S}.
The set $\cM(K, N, D)$ of closed Riemannian manifolds with Ricci curvature $\ge K$, dimension $\le N$, and diameter $\le D$ is a measured GH precompact and hence $\Box$-precompact,
cf.~\cite[Remarks~4.32, 4.34]{S}.
Here and hereafter, GH stands for Gromov--Hausdorff.
\end{ex}

We remark that the \emph{pyramid}
\begin{align}\label{eq;pyramid}
\cP_X := \{ Y\in\cX : Y\prec X \}
\end{align}
associated with an mm-space $X$ is a compact set in $(\cX, \Box)$,
e.g.~\cite[Chapter~6]{S},
and the converse of Theorem~\ref{thm;main} also holds,
cf.~Corollary~\ref{cor;main}.

For metric spaces $X$ and $Y$,
we write $Y\prec X$ if there exists a $1$-Lipschitz surjection $f:X\to Y$.
The following is a byproduct of this work and a GH analog of Theorem~\ref{thm;main}.
\begin{prop}\label{prop;GH}
If a set $\cY$ of compact metric spaces is GH-precompact,
then there exists a compact metric space $X$ with $Y\prec X$ for any $Y\in\cY$.
\end{prop}

After some preparations in Section~\ref{sec;pre},
we present a proof of Theorem~\ref{thm;main} in Section~\ref{sec;pf}.
Our proof is a slight modification of Gromov's sketch in \cite{G}.
In Section~\ref{sec;GH}, we prove Proposition~\ref{prop;GH}.

In our argument,
we do not pay careful attention to distinguish an mm-space and an element of $\cX$
and consider an element $x$ of a product $\prod_{\lambda\in\Lambda} X_\lambda$ of a family $\{X_\lambda\}_{\lambda\in\Lambda}$ of sets
as a map $x:\Lambda\to\bigcup_{\lambda\in\Lambda} X_\lambda$
with $x(\lambda)\in X_\lambda$ for any $\lambda\in\Lambda$.

\section{Preliminaries}\label{sec;pre}

In this section, we recall some definitions and facts from \cite{S}.
A reader who is familiar with it can safely skip this section.
We will use the notations in \cite{S} such as
\[
U_\ep (A) := \{ x\in X : d_X (x, A) <\ep \}
\]
for a subset $A\subset X$ of a metric space $(X, d_X)$ and $\ep>0$.

\begin{defn}[e.g.~{\cite[Definition~4.21]{S}}]
Let $X$ and $Y$ be mm-spaces and $\ep>0$.
We call a Borel map $f:X\to Y$ an \emph{$\ep$-mm-isomorphism} if
there exists a Borel set $\tX\subset X$, called a \emph{non-exceptional domain}, with
\begin{enumerate}
\item $\mux(\tX) \ge 1-\ep$,
\item $|d_X (x, x') -d_Y (f(x), f(x'))| \le \ep$ for any $x, x'\in\tX$, and
\item $\proh(f_*\mux, \muy) \le\ep$,
\end{enumerate}
where $\proh$ denotes the Prohorov distance between Borel probability measures on $(X, d_X)$.
\end{defn}

The definition of the box distance $\Box$ can be found in e.g.~\cite[Definition~4.4]{S}.
We know that $(\cX, \Box)$ is a complete metric space, e.g.~\cite[Theorems~4.10, 4.14]{S}.
We need the following rather than its precise definition.
\begin{lem}[e.g.~{\cite[Lemma~4.22]{S}}]\label{lem;mmisombox}
Let $X$ and $Y$ be mm-spaces and $\ep>0$.
\begin{enumerate}
\item If there exists an $\ep$-mm-isomorphism $f:X\to Y$, then $\Box(X, Y)\le 3\ep$.
\item If $\Box(X, Y)<\ep$, then there exists a $3\ep$-mm-isomorphism $f:X\to Y$.
\end{enumerate}
\end{lem}

We summarize a lemma in \cite{S} as follows.
\begin{lem}[e.g.~{\cite[Lemma~4.28]{S}}]\label{lem;precpt}
Let $\cY\subset\cX$.
Then $\cY$ is $\Box$-precompact if and only if
there exists $\Delta=\Delta(\ep)<\infty$ for any $\ep>0$ such that
any $Y\in\cY$ admits a set $\cK_Y$ of Borel sets of $Y$ with
\[
\#\cK_Y \le\Delta, \quad
\max_{K\in\cK_Y} \diam K \le\ep, \quad
\diam \bigcup\cK_Y \le\Delta, \quad\text{ and }\quad
\muy(\bigcup\cK_Y)\ge 1-\ep.
\]
\end{lem}

\begin{defn}[{\cite[Definitions~6.3, 6.4]{S}}]
A \emph{pyramid} is a non-empty closed set $\cP$ in $(\cX, \Box)$ with the following properties:
\begin{enumerate}
\item If $X\in\cP$ and $Y\in\cX$ satisfy $Y\prec X$, then $Y\in\cP$.
\item If $X, Y\in\cP$, then there exists $Z\in\cP$ with $X, Y\prec Z$.
\end{enumerate}

We say that a sequence $\seqn{\cP_n}$ of pyramids \emph{converges weakly} to a pyramid $\cP$ if
\begin{enumerate}
\item $\lim_{n\to\infty} \Box(X, \cP_n)=0$ for any $X\in\cP$ and
\item $\liminf_{n\to\infty} \Box(Y, \cP_n)>0$ for any $Y\in\cX\setminus\cP$.
\end{enumerate}
\end{defn}

A typical example of pyramids is $\cP_X$ associated with an mm-space $X$ in \eqref{eq;pyramid}.
The only fact that we need is that
the sequence $\seqn{\cP_{X_n}}$ of pyramids converges weakly to $\cP_Y$
if a sequence $\seqn{X_n}$ of mm-spaces converges to an mm-space $Y$ in $(\cX, \Box)$,
e.g.~\cite[Propositions~5.5, 6.13]{S}.

\section{Proof of Theorem~\ref{thm;main}}\label{sec;pf}

In this section, we make some preparations and present a proof of Theorem~\ref{thm;main}.

\begin{defn}
\label{def;CDLip}
Let $C, D\ge 0$.
We say that a map $f:X\to Y$ between metric spaces $X$ and $Y$ is \emph{$(C, D)$-Lipschitz} if it satisfies
\[
d_Y (f(x), f(x'))\le C d_X (x, x')+D
\]
for any $x, x' \in X$.
A $(C, 0)$-Lipschitz map is simply called a \emph{$C$-Lipschitz map}.
\end{defn}

\begin{lem}[cf.~{\cite[Lemma~4.19]{S}}]\label{lem;P}
Let $X$ be an mm-space and $\ep, \ep'>0$.
If a set $\cA$ of disjoint Borel sets of $X$ and a Borel probability measure $\nu$ on $X$ satisfy
\[
\sup_{A\in\cA} \diam A \le\ep, \qquad
\mux(C)\le\ep, \qquad
\sum_{A \in \cA} |\mux(A)-\nu(A)| \le\ep',
\]
where $C:= X\setminus\bigcup\cA$,
then $\proh(\mux, \nu)\le \ep+\ep'$.
\end{lem}

\begin{proof}
Take any Borel set $B$ of $X$ and let $\cA' := \{A \in \cA : A \cap B \neq \emptyset\}$ and $\overline{\ep}:=\ep+\ep'$.
Then we have
\[
\mux(B) \le \sum_{A \in \cA'} \mux(A) + \mux(C) \le \sum_{A \in \cA'} \nu(A) +\overline{\ep} \le \nu(U_{\overline{\ep}} (B))+\overline{\ep}
\]
and obtain $\proh(\mux, \nu)\le\overline{\ep}$.
\end{proof}

In applying Lemma~\ref{lem;P},
we note that $1-\ep\le \nu(A)/\mux(A) \le 1+\ep$ for any $A\in\cA$ implies
\[
\sum_{A \in \cA} |\mux(A)-\nu(A)| \le \ep\sum_{A \in \cA} \mux(A) \le\ep.
\]

\begin{defn}[cf.~\cite{N}]
Let $X, Y$ be mm-spaces, $\cY\subset\cX$ a subset, and $\ep>0$.

We write $\cY\prec X$ if $Y\prec X$ for any $Y\in\cY$.

We write $Y\prec_\ep X$ if
there exist a Borel set $\tX\subset X$, called a \emph{non-exceptional domain}, and a Borel map $f:X\to Y$ with
\begin{enumerate}
\item $\mux(\tX)\ge 1-\ep$,
\item $d_Y (f(x), f(x')) \le d_X (x, x') +\ep$ for any $x, x'\in \tX$, and
\item $\proh(f_* \mux, \muy) \le\ep$.
\end{enumerate}

We write $\cY\prec_{\ep, 0} X$ if the condition $Y\prec_\ep X$ holds with $\tX=X$ for any $Y\in\cY$.
\end{defn}

\begin{lem}\label{lem;O}
Let $Y$ be an mm-space and $\ep>0$.
Suppose that $\cA$ is a finite set of disjoint Borel sets of $Y$ with
\[
\diam A <\ep,\ \muy(\partial A)=0, \text{ and } \muy(A) >0 \text{ for any } A\in\cA.
\]
Then there exists an open neighborhood $\cO=\cO(Y, \cA, \ep)\subset\cX$ of $Y$ in $(\cX, \Box)$ for which
any $Z\in\cO$ admits a Borel map $\Phi:Y\to Z$ and a finite set $\cA_Z$ of disjoint Borel sets of $Z$ with
\begin{enumerate}
\item $| d_Y (y, y') -d_Z (\Phi(y), \Phi(y')) |< 3\ep \text{ for any } y, y' \in U$, where $U:=\bigcup\cA$,
\item $\diam A<\ep$, $\muz(A)>0$, $\Phi^{-1} (A) \cap U \in\cA$, and
\[
1-\ep < \frac{\Phi_* \muy(A)}{\muz (A)} < 1+\ep \text{ for any $A\in\cA_Z$, and}
\]
\item $\Phi(U)\subset U_Z$, where $U_Z :=\bigcup\cA_Z$.
\end{enumerate}
\end{lem}

\begin{proof}
If such a neighborhood $\cO(Y, \cA, \ep)$ does not exist,
there exists a sequence $\seqn{Z_n}$ of mm-spaces with $\Box(Z_n, Y)\to 0$ as $\nti$ for which
$Z_n$ does not admit such $\cA_{Z_n}$.
Then there exists an $\ep_n$-mm-isomorphism $\Psi_n:Z_n \to Y$ with a non-exceptional domain $\tZ_n \subset Z_n$ for each $n$ with $\ep_n \to 0$ as $\nti$ by Lemma~\ref{lem;mmisombox}.
We define
\[
\cA_{Z_n} := \{ \Psi_n^{-1} (A) \cap \tZ_n : A\in\cA \} \quad\text{ and }\quad U_{Z_n} :=\bigcup\cA_{Z_n}.
\]
Note that $\diam (\Psi_n^{-1} (A) \cap \tZ_n) \le \diam A + \ep_n < \ep$ for any $A \in \cA$ and large $n$.

Since $(\Psi_n)_* \mu_{Z_n}$ weakly converges to $\muy$, the Portmanteau theorem implies
\begin{align}\label{eq;Portmanteau}
\lim_{n\to\infty} \mu_{Z_n}(\Psi_n^{-1}(A)\cap \tZ_n) = \lim_{n\to\infty} (\Psi_n)_* \mu_{Z_n}(A) = \muy(A) > 0
\end{align}
and we choose a point $\pt_{A, n} \in \Psi_n^{-1}(A)\cap \tZ_n$ for each $A\in\cA$.
We also note that
\[
\lim_{n\to\infty} \mu_{Z_n}(U_{Z_n})
= \lim_{n\to\infty} \mu_{Z_n}( \Psi_n^{-1} (U))
= \lim_{n\to\infty} (\Psi_n)_* \mu_{Z_n}(U)
= \muy(U)
\]
and choose some $z_n \in Z_n$ if $\muy(U)=1$ and $z_n \in Z_n \setminus U_{Z_n}$ if $\muy(U)<1$.

Define a map $\Phi_n:Y \to Z_n$ as
\[
\Phi_n (y) := \left\{\begin{array}{ll} \pt_{A, n} & \text{if } y \in A, \\
z_n & \text{if } y \in Y\setminus U. \end{array}\right.
\]
Then we have $\Phi_n (U) \subset U_{Z_n}$ and
\begin{align*}
| d_Y (y, y') -d_{Z_n} (\Phi_n(y), \Phi_n(y')) |
& = | d_Y (y, y') -d_{Z_n} (\pt_{A, n}, \pt_{A', n}) | \\
& \le | d_Y (y, y') -d_Y (\Psi_n (\pt_{A, n}), \Psi_n (\pt_{A', n})) | + \ep_n \\
& \le d_Y (y, \Psi_n (\pt_{A, n})) + d_Y (y', \Psi_n (\pt_{A', n})) + \ep_n
< 3\ep
\end{align*}
for any $y \in A$ and any $y' \in A'$ with $A, A' \in \cA$.

Moreover we have $\muy(\Phi_n^{-1}(\Psi_n^{-1}(A)\cap \tZ_n)) = \muy(A)$
and Equation~\eqref{eq;Portmanteau} implies
\begin{align*}
\lim_{n\to\infty}  \frac{(\Phi_n)_* \muy(\Psi_n^{-1}(A)\cap \tZ_n)}{\mu_{Z_n}(\Psi_n^{-1}(A)\cap \tZ_n)}
=\frac{\muy(A)}{\muy(A)} =1
\end{align*}
for every $A\in\cA$.
These contradict the assumption and finish the proof.
\end{proof}

\begin{cor}\label{cor;O}
Let $Y$ be an mm-space and $\ep>0$.
Suppose that $\cA$ is as in Lemma~\ref{lem;O}.
Then there exists an open neighborhood $\cO=\cO(Y, \cA, \ep)\subset\cX$ of $Y$ in $(\cX, \Box)$ with the following properties:
If $X$ is an mm-space and $h:X\to Y$ is a Borel map with $h(X) \subset U:=\bigcup\cA$,
then any $Z\in\cO$ admits a Borel map $h_Z: X\to Z$ and a finite set $\cA_Z$ of disjoint Borel sets of $Z$ with
\begin{enumerate}
\item $\diam A<\ep$ and $\muz(A)>0$ for any $A\in\cA_Z$,
\item $| d_Y (h(x), h(x')) -d_Z (h_Z (x), h_Z (x')) |< 3\ep \text{ for any } x, x' \in h^{-1}(U)$,
\item $h_Z (X) \subset U_Z$, where $U_Z :=\bigcup\cA_Z$, and
\item there is a bijection $\cA_Z \to \cA; A\mapsto A_Y$ for which any $A\in\cA_Z$ satisfies
\[
(h_Z)_* \mux (A) =h_*\mux (A_Y) \quad\text{ and }\quad 1-\ep< \frac{\muy(A_Y)}{\muz(A)} <1+\ep.
\]
\end{enumerate}
\end{cor}

\begin{proof}
Let $\cO=\cO(Y, \cA, \ep)$ be as in Lemma~\ref{lem;O}.
For $Z\in\cO$, let $\Phi :Y\to Z$ and $\cA_Z$ be also as in Lemma~\ref{lem;O} and define $h_Z := \Phi\circ h:X\to Z$.

If $x, x' \in X$ and $y=h(x), y'=h(x')\in U$, we have
\[
| d_Y (h(x), h(x')) -d_Z (h_Z (x), h_Z (x')) |
=| d_Y (y, y') -d_Z (\Phi(y), \Phi(y')) |< 3\ep.
\]
For any $A \in \cA_Z$, the set $A_Y := \Phi^{-1}(A) \cap U \in\cA$ satisfies
\[
(h_Z)_* \mux (A) =\mux(h^{-1}(\Phi^{-1}(A)))= h_*\mux (A_Y) \quad\text{ and }\quad
\frac{\muy(A_Y)}{\muz(A)} = \frac{\Phi_* \muy(A)}{\muz(A)}.
\]
Thus $h_Z$ has the desired properties by Lemma~\ref{lem;O}.
\end{proof}

We use the following in the first step of our proof of Theorem~\ref{thm;main}.
\begin{lem}\label{lem;step1}
Let $\cY\subset\cX$ be a finite set and
$\cA_Y$ a non-empty finite set of disjoint Borel sets of $Y\in\cY$ with $\muy(A)>0$ for any $A\in\cA_Y$ and $U_Y:=\bigcup\cA_Y \subset Y$.
Then there exist a finite mm-space $X$ and $1$-Lipschitz maps $f_Y : X\to Y$ with
\[
f_Y (X)\subset U_Y \quad\text{ and }\quad  (f_Y)_* \mux(A)=\frac{\muy(A)}{\muy(U_Y)}
\]
for any $Y\in\cY$ and $A\in\cA_Y$.
\end{lem}

\begin{proof}
We choose $\pt_A\in A$ for each $A\in\bigcup_{Y\in\cY} \cA_Y$ and define
$X:= \prod_{Y\in\cY} \cA_Y$,
\[
d_X (x, x') := \max_{Y\in\cY} d_Y (\pt_{x(Y)}, \pt_{x'(Y)}), \quad\text{ and }\quad
\mux(\{x\}) := \prod_{Y\in\cY} \frac{\muy(x(Y))}{\muy(U_Y)}
\]
for $x, x' \in X$.
Then $X=(X, d_X, \mux)$ is a finite mm-space and the map
\[
f_Y: X\to Y, \quad f_Y (x)=\pt_{x(Y)}
\]
is $1$-Lipschitz and satisfies
\[
(f_Y)_* \mux(A)=\mux(f_Y^{-1} (A))=\mux(f_Y^{-1} (\pt_A)) =\frac{\muy(A)}{\muy(U_Y)}
\]
for any $Y\in\cY$ and $A\in\cA_Y$.
This finishes the proof.
\end{proof}

Although we do not use the following in our argument,
we record it for possible future applications.
\begin{cor}
If $\cY\subset\cX$ is precompact in $(\cX, \Box)$ and $\ep>0$,
then there exists a finite mm-space $X$ with $\cY\prec_{\ep, 0} X$.
\end{cor}

\begin{proof}
Since any mm-space $X$ satisfies $\cY\prec_{\ep, 0} X$ if $\ep\ge 1$,
we may assume that $\ep<1$.
For each $Y\in\overline{\cY}$,
we take a finite set $\cA_Y$ of disjoint Borel sets of $Y$ with
\[
\diam A <\ep, \ \muy(\partial A)=0, \ \muy(A)>0 \text{ for any } A\in\cA_Y \text{ and } \muy(U_Y)> (1+\ep)^{-1},
\]
where $U_Y := \bigcup\cA_Y$,
cf.~\cite[Lemma~42]{OY}.
Since $\cY$ is $\Box$-precompact,
Corollary~\ref{cor;O} implies that there exists a finite subset $\cY' \subset\overline{\cY}$ with $\cY \subset \bigcup_{Y\in\cY'} \cO(Y, \cA_Y, \ep)$ and
Lemma~\ref{lem;step1} states that
there exist a finite mm-space $X$ and $1$-Lipschitz maps $f_Y :X\to Y$ with
\[
f_Y (X)\subset U_Y \quad\text{ and }\quad
\frac{(f_Y)_* \mux(A)}{\muy(A)}=\frac{1}{\muy(U_Y)}
\]
for any $Y\in\cY' \text{ and } A\in\cA_Y$.
We note that $\muy(U_Y) > (1+\ep)^{-1} >1-\ep$.

For any $Z\in\cO(Y, \cA_Y, \ep)$ with some $Y\in\cY'$,
Corollary~\ref{cor;O} implies that there exists a map $f_Z: X\to Z$ with
\[
d_Z (f_Z (x), f_Z (x')) \le d_Y (f_Y (x), f_Y (x')) +3\ep \le d_X (x, x') +3\ep
\]
for any $x, x' \in X$ and, if we replace $\cA_Z$ with the one in Corollary~\ref{cor;O},
we have $(1+\ep)^2 <1+3\ep$,
\[
1-\ep < \frac{(f_Z)_* \mux(A)}{\muz(A)}
= \frac{(f_Y)_* \mux(A_Y)}{\muy(A_Y)}\frac{\muy(A_Y)}{\muz(A)}
= \frac{1}{\muy(U_Y)}\frac{\muy(A_Y)}{\muz(A)}
< 1+3\ep
\]
for any $A\in\cA_Z$, and $\muz(U_Z) > (1-\ep)\muy(U_Y) > 1-2\ep$.

Then Lemma~\ref{lem;P} yields $\proh(\muy, (f_Y)_* \mux) \le 5\ep$ for any $Y\in\cY$.
Thus $\cY\prec_{5\ep, 0} X$ and this finishes the proof.
\end{proof}

The following two lemmas are the key technical lemmas in our proof of Theorem~\ref{thm;main}.
\begin{lem}\label{lem;key}
Let $X$ be a finite mm-space, $\cY\subset\cX$ a finite subset, and $\ep, \ep'>0$.
Suppose that any $Y\in\cY$ admits a map $f_Y: X\to Y$ and finite sets $\cA_Y$ and $\cA'_Y$ of disjoint Borel sets of $Y$ with
\begin{enumerate}
\item $d_Y (f_Y (x), f_Y (x')) \le d_X (x, x') +\ep$ for any $x, x' \in X$,
\item $f_Y (X)\subset U_Y$, where $U_Y :=\bigcup\cA_Y$,
\item $\diam A\le\ep$ and $(f_Y)_* \mux(A) =\muy(A)/\muy(U_Y)$ for any $A\in\cA_Y$,
\item $\muy(Y \setminus U'_Y) < (\ep'/(1+\ep')) \min\{ \muy(A) : A\in\cA_Y \}$, where $U'_Y :=\bigcup\cA'_Y$,
\item $\muy(B)>0$ for any $B\in\cA'_Y$, and
\item $B\subset A$ or $A\cap B=\emptyset$ if $A\in\cA_Y$ and $B\in\cA'_Y$.
\end{enumerate}
Then there exists a finite mm-space $X'$ with $\Box(X, X') \le 9\ep$ for which
any $Y\in\cY$ admits a $1$-Lipschitz map $g_Y: X' \to Y$ with $g_Y (X)\subset U_Y \cap U'_Y$ and
\begin{align}\label{eq;key}
1\le \muy(U_Y) \frac{(g_Y)_* \mu_{X'} (B)}{\muy(B)} <1+\ep' \text{ for any } B\in\cA'_Y \text{ with } B\subset U_Y.
\end{align}
\end{lem}

\begin{proof}
For any $x\in X$ and $Y\in\cY$, there exists a unique $A_{Y, x} \in\cA_Y$ with $f_Y (x) \in A_{Y, x}$.
We put
\[
\cA'(A) := \{ B\in\cA'_Y : B\subset A \}
\]
for a subset $A\subset Y$
and define
\[
X'_x := \prod_{Y\in\cY} \cA'(A_{Y, x}) \text{ for } x\in X \quad\text{ and }\quad X' := \bigsqcup_{x\in X} X'_x.
\]

We choose a point $\pt_B\in B$ for each $B\in\bigcup_{Y\in\cY} \cA'_Y$ and define
\begin{align*}
d_{X'} (\phi, \phi') &:= \max\left\{ \max_{Y\in\cY} d_Y (\pt_{\phi(Y)}, \pt_{\phi'(Y)}), \, d_X (x, x') \right\},\\
\mu_{X'} (\{\phi\}) &:= \mux(\{x\}) \frac{ \prod_{Y\in\cY} \muy(\phi(Y)) }{ \prod_{Y\in\cY} \muy(\bigcup\cA'(A_{Y, x})) }
\end{align*}
for $\phi\in X'_x$ and $\phi' \in X'_{x'}$ with $x, x' \in X$.
Then $X'=(X', d_{X'}, \mu_{X'})$ is a finite mm-space, because
$\mu_{X'} (X'_x) =\mux(\{x\})$ for any $x\in X$ and
\[
\mu_{X'} (X') = \sum_{x\in X} \mu_{X'} (X'_x) =\sum_{x\in X} \mux(\{x\})=1.
\]

For any $Y\in\cY$,
the map $g_Y :X'\to Y$ defined by $g_Y (\phi) = \pt_{\phi(Y)}$ is $1$-Lipschitz and satisfies $g_Y (X)\subset U_Y \cap U'_Y$.

The natural map $\pi:X' \to X$ defined by $\pi(\phi)=x$ if $x\in X$ and $\phi\in X'_x$ satisfies $\pi_* \mu_{X'} =\mux$.
For any $Y\in\cY$, $\phi\in X'_x$, and $\phi' \in X'_{x'}$ with $x, x' \in X'$,
we also have
\begin{align*}
d_Y (\pt_{\phi(Y)}, \pt_{\phi'(Y)})
&\le d_Y (\pt_{\phi(Y)}, f_Y (x)) + d_Y (f_Y (x), f_Y (x')) + d_Y (f_Y (x'), \pt_{\phi'(Y)}) \\
&\le d_X (x, x') +3\ep
\end{align*}
and hence
\begin{align}\label{ineq;3ep}
d_X (x, x')
\le d_{X'} (\pi(x), \pi(x'))
= d_{X'} (\phi, \phi')
\le d_X (x, x') +3\ep.
\end{align}
Thus Lemma~\ref{lem;mmisombox} yields $\Box(X, X')\le 9\ep$.

Finally, we have
\[
\muy(A)-\muy(\bigcup\cA'(A))=\muy(A\setminus U'_Y) \le \muy(Y\setminus U'_Y) < \frac{\ep'}{1+\ep'} \muy(A),
\]
and
\[
(g_Y)_* \mu_{X'} (B)
= \mu_{X'} (g_Y^{-1}(B))
=\mux(f_Y^{-1}(A)) \frac{\muy(B)}{\muy(\bigcup\cA'(A))}
=\frac{\muy(A)}{\muy(U_Y)} \frac{\muy(B)}{\muy(\bigcup\cA'(A))}
\]
for any $A\in\cA_Y$ and $B\in\cA'(A)$.
Thus Inequality~\eqref{eq;key} follows.
This completes the proof of Lemma~\ref{lem;key}.
\end{proof}

\begin{lem}\label{lem;X'}
Let $X$ be a finite mm-space, $\cY\subset\cX$ a finite subset, and $\ep, \ep'>0$ with $\ep < 1$.
Suppose that any $Y\in\cY$ admits a map $g_Y: X\to Y$ and a finite set $\cA_Y$ of disjoint Borel sets of $Y$ with
\begin{enumerate}
\item $g_Y (X)\subset U_Y$, where $U_Y :=\bigcup\cA_Y$,
\item $d_Y (g_Y (x), g_Y (x')) \le d_X (x, x')+\ep'$ for any $x, x' \in X$, and
\item $\muy(A)>0$ and $(g_Y)_* \mux(A) \le (1+\ep)\muy(A)$ for any $A\in\cA_Y$.
\end{enumerate}
Then there exists a finite mm-space $X'$ with $\Box(X, X')\le 3\ep$ for which
any $Y\in\cY$ admits a map $f_Y :X'\to Y$ with
\begin{enumerate}
\setcounter{enumi}{3}
\item $f_Y (X')\subset U_Y$,
\item\label{item;epLip} $d_Y (f_Y (x), f_Y (x')) \le d_{X'} (x, x')+\ep'$ for any $x, x' \in X'$, and
\item $(f_Y)_* \mu_{X'} (A) = \muy(A)/\muy(U_Y)$ for any $A\in\cA_Y$.
\end{enumerate}
\end{lem}

\begin{proof}
We fix $x_0\in X$ and choose some $\pt_A \in A$ for each $A\in\cA_Y$ with
\[
g_Y (x_0) \in Y', \text{ where } Y':= \{ \pt_A : A\in\cA_Y \} \text{ for any } Y\in\cY.
\]
Then we define a probability measure $\mu_{Y'}$ on $Y'$ by
\[
\mu_{Y'} (\{\pt_A\}) := \frac{1}{\ep}\left(\frac{\muy(A)}{\muy(U_Y)}-(1-\ep)(g_Y)_* \mux(A)\right)
\]
for $A\in\cA_Y$ and note that
\[
\mu_{Y'} (\{\pt_A\}) \ge \frac{1}{\ep}(1-(1-\ep)(1+\ep))\muy(A) =\ep\muy(A)>0
\]
for any $A\in\cA_Y$ and $\mu_{Y'} (Y') = 1$.
Thus $(Y', d_Y, \mu_{Y'})$ is a finite mm-space.

We equip $X'' := \prod_{Y\in\cY} Y'$ with the $\ell_\infty$-product distance $d_{X''}$
and define
\[
d_{X'} (x, x') = d_{X'} (x', x) :=
\begin{cases}
\ d_Z (x, x') & \text{ if } x, x' \in Z=X \text{ or } X'',  \\
\ d_X (x, x_0) + \diam X'' & \text{ if } x\in X \text{ and } x' \in X''.
\end{cases}
\]
Then we define a probability measure $\mu_{X'}$ on $X' := X\sqcup X''$ by
\[
\mu_{X'}|_{X} =(1-\ep)\mux \quad\text{ and }\quad \mu_{X'}|_{X''} = \ep\bigotimes_{Y\in\cY} \mu_{Y'}
\]
to make $X'$ an mm-space.

The inclusion $\iota : X\hookrightarrow X'$ is an $\ep$-mm-isomorphism because it is isometric,
\[
\iota_* \mux(B) = \mux(B \cap X) \le (1-\ep)\mux(B \cap X) + \ep \le \mu_{X'}(B) + \ep
\]
for any Borel set $B\subset X'$, and hence $\proh(\iota_* \mux, \mu_{X'})\le\ep$.
Thus Lemma~\ref{lem;mmisombox} yields $\Box(X, X')\le 3\ep$.

We fix $Y\in\cY$ and define a map $f=f_Y: X'\to Y$ by $f|_X = g_Y$ and $f|_{X''} = \pr_{Y'}$,
where $\pr_{Y'}: X''\to Y'$ is the projection.
In order to check Condition~\eqref{item;epLip},
we only have to verify it for $x \in X$ and $x' \in X''$ and actually have
\begin{align*}
 d_Y (f(x), f(x'))
& = d_Y(g_Y(x), \pr_Y(x')) \\
& \le d_Y(g_Y(x), g_Y (x_0)) + d_Y(g_Y (x_0), \pr_Y(x')) \\
& \le d_X(x, x_0) +\ep' + \diam{X''}
= d_{X'} (x, x')+\ep'.
\end{align*}
We also have $f(X') \subset U_Y$ and
\[
f_* \mu_{X'} (A) = (1-\ep) (g_Y)_*\mux (A) + \ep \mu_{Y'} (\pr_{Y'}^{-1} (A)) =\frac{\muy(A)}{\muy(U_Y)}
\]
for any $A\in\cA_Y$.
This completes the proof of Lemma~\ref{lem;X'}.
\end{proof}

We use the following in the final step of our proof of Theorem~\ref{thm;main}.
\begin{lem}[cf.~{\cite[Lemma~4.39]{S}}]\label{lem;final}
Let $\seqn{\ep_n}$ be a sequence of positive numbers with $\ep_n \to 0$ as $\nti$.
If sequences $\seqn{X_n}$ and $\seqn{Y_n}$ of mm-spaces $\Box$-converge to mm-spaces $X$ and $Y$ respectively and $Y_n \prec_{\ep_n} X_n$ for any $n$,
then $Y\prec X$.
\end{lem}

While Lemma~\ref{lem;final} follows from Lemma~\ref{lem;mmisombox} and \cite[Lemma~4.39]{S},
cf.~\cite{N},
it also follows from Proposition~\ref{prop;indep} below.

\begin{defn}
For pyramids $\cP, \cQ$ and $\ep>0$,
we write $\cQ\prec_\ep \cP$ if
any $Y\in\cQ$ admits $X\in\cP$ with $Y\prec_\ep X$.
\end{defn}

\begin{prop}\label{prop;indep}
Let $\seqn{\ep_n}$ be a sequence of positive numbers with $\ep_n \to 0$ as $\nti$.
If sequences $\seqn{\cP_n}$ and $\seqn{\cQ_n}$ of pyramids converge weakly to pyramids $\cP$ and $\cQ$ respectively
and $\cQ_n \prec_{\ep_n} \cP_n$ for any $n$, then $\cQ\subset\cP$.
\end{prop}

To prove Proposition~\ref{prop;indep}, we use the following.
\begin{lem}[{\cite[Lemma~4.6]{K}}]\label{lem;K}
If $X$ and $Y$ are mm-spaces with $Y \prec_\ep X$ for some $\ep > 0$,
then there exists an mm-space $Z$ with $Z\prec X$ and $\Box(Y, Z) \le 4\ep$.
\end{lem}

\begin{proof}[Proof of Proposition~\ref{prop;indep}]
Let $Y\in\cQ$.
Then a sequence $\seqn{Y_n}$ of mm-spaces with $Y_n \in\cQ_n$ for any $\niN$ $\Box$-converges to $Y$.
By assumption,
there exists $X_n \in\cP_n$ with $Y_n \prec_{\ep_n} X_n$ for each $\niN$.
By Lemma \ref{lem;K}, for each $\niN$, there exists an mm-space $Z_n$ with $Z_n \prec X_n$ and $\Box(Y_n, Z_n) \le 4\ep_n$. Note that $Z_n \in \cP_n$ for each $\niN$ and $\seqn{Z_n}$ $\Box$-converges to $Y$.
Thus the weak convergence of $\seqn{\cP_n}$ implies $Y \in \cP$.
This means that $\cQ\subset\cP$.
\end{proof}

\begin{cor}[cf.~{\cite[Proposition~2.11]{S}}]
If pyramids $\cP$ and $\cQ$ satisfy $\cP\prec_\ep \cQ$ and $\cQ\prec_\ep \cP$ for any $\ep>0$,
then $\cP=\cQ$.
\end{cor}

\begin{proof}
We have $\cP\subset\cQ$ and $\cQ\subset\cP$ by Proposition~\ref{prop;indep} and hence $\cP=\cQ$.
\end{proof}

Now we are in a position to prove Theorem~\ref{thm;main}.
\begin{proof}[Proof of Theorem~\ref{thm;main}]
We suppose that $\cY\subset\cX$ is precompact in $(\cX, \Box)$ and
fix a decreasing sequence $\seqn{\ep_n}$ of positive numbers $\ep_n>0$ with $\ep_1 <1$ and $\sum_\niN \ep_n <\infty$.

For each $\niN$,
we take a finite set $\cA_{Y, n}$ of Borel sets of $Y\in\overline{\cY}$ and a finite subset $\cY_n \subset\overline{\cY}$ with
\begin{enumerate}
\item $\diam A<\ep_n$, $\muy(A)>0$, and $\muy(\partial A)=0$ for any $A\in\cA_{Y, n}$,
\item $\cO_{Y, n} := \cO(Y, \cA_{Y, n}, \ep_n)$ is as in Corollary~\ref{cor;O} for $Y\in\cY_n$ and $\bigcup_{Y\in\cY_n} \cO_{Y, n} =\cY$,
\item $\cA_{Z, n}$ is as in Corollary~\ref{cor;O} if $Y\in\cY_n$ and $Z\in\cO_{Y, n}$,
\item $B\subset A$ or $A\cap B=\emptyset$ if $Y\in\cY_{n+1}$, $A\in\cA_{Y, n}$, and $B\in\cA_{Y, n+1}$,
\item $\muy(Y\setminus\bigcup\cA_{Y, n+1})<(\ep_{n+1}/(1+\ep_{n+1})) \min\{ \muy(A) : A\in\cA_{Y, n} \}$ for any $Y\in\cY_{n+1}$,
\item $\cY_n \subset\cY_{n+1}$, and $\cY\subset\overline{\bigcup_{n=1}^\infty \cY_n}$,
\end{enumerate}
cf.~\cite[Lemma~42]{OY}.

We note that $U_{Y, n} :=\bigcup\cA_{Y, n} \subset Y$ satisfies $\muy(U_{Y, n})> (1 +\ep_n)^{-1} >1-\ep_n$ for any $Y\in\cY_n$.

\begin{clm}\label{clm}
For any $\niN$,
there exist a finite mm-space $X_n$ with $\Box(X_n, X_{n+1}) \le 48\ep_n$ and
$(1, 3\ep_n)$-Lipschitz maps $f_{Y, n} :X_n \to Y$ in the sense of Definition~\ref{def;CDLip} satisfying
\begin{align}\label{eq;fn}
f_{Y, n} (X_n) \subset U_{Y, n} \quad\text{ and }\quad
(f_{Y, n})_* \mu_{X_n} (A)=\frac{\muy(A)}{\muy(U_{Y, n})}
\end{align}
for any $Y\in\cY_{n+1} \text{ and } A\in\cA_{Y, n}$.
\end{clm}

\begin{proof}
We start by using Lemma~\ref{lem;step1} to obtain a finite mm-space $X_1$ and
$1$-Lipschitz maps $f_{Y, 1}:X_1 \to Y$ satisfying \eqref{eq;fn} with $n=1$ for any $Y\in\cY_2$.
Then we suppose that we have a finite mm-space $X_k$ and
$(1, 3\ep_k)$-Lipschitz maps $f_{Y, k}:X_k \to Y$ satisfying \eqref{eq;fn} with $n=k\in\N$ for any $Y\in\cY_{k+1}$ and $A\in\cA_{Y, k}$.

We use Lemma~\ref{lem;key} to obtain a finite mm-space $X'_k$ with $\Box(X_k, X'_k) \le 27\ep_{k}$ and
$1$-Lipschitz maps $g_{Y, k}: X'_k:\to Y$ with
\[
g_{Y, k} (X'_k) \subset U_{Y, k+1} \quad\text{ and }\quad
\muy(U_{Y, k}) \frac{(g_{Y, k})_* \mu_{X'_k} (B)}{\muy(B)} <1+\ep_{k+1}
\]
for any $Y\in\cY_{k+1} \text{ and } B\in\cA_{Y, k+1}$.

Then we use Corollary~\ref{cor;O} to obtain $(1, 3\ep_{k+1})$-Lipschitz maps $g_{Z, k}: X'_k \to Z$ with
\[
g_{Z, k} (X'_k) \subset U_{Z, k+1} \quad\text{ and }\quad
\frac{(g_{Z, k})_* \mu_{X'_k} (B)}{\muz(B)} <(1+\ep_k)^3 <1+ 7\ep_{k}
\]
for any $Z\in\overline{\cY}\setminus\cY_{k+1} \text{ and } B\in\cA_{Z, k+1}$.

Finally, we use Lemma~\ref{lem;X'} to obtain a finite mm-space $X_{k+1}$ with $\Box(X'_k, X_{k+1}) \le 21\ep_k$
and $(1, 3\ep_{k+1})$-Lipschitz maps $f_{Y, k+1}: X_{k+1} \to Y$ satisfying \eqref{eq;fn} with $n=k+1$
for any $Y\in\cY_{k+2} \text{ and } A\in\cA_{Y, k+1}$.

Then
\[
\Box(X_k, X_{k+1}) \le \Box(X_k, X'_k) +\Box(X'_k, X_{k+1}) \le 48\ep_k
\]
and Claim~\ref{clm} is verified.
\end{proof}

We obtained a sequence $\seqn{X_n}$ of mm-spaces and maps $f_{Y, n} :X_n \to Y$ as in Claim~\ref{clm}.
Lemma~\ref{lem;P} yields $\cY_{n+1} \prec_{3\ep_n, 0} X_n$ for any $\niN$.
Since $\sum_\niN \Box(X_n, X_{n+1}) <\infty$,
the sequence $\seqn{X_n}$ is Cauchy and converges to some mm-space $X$ in $(\cX, \Box)$.
Then $\cY\prec X$ by Lemma~\ref{lem;final}.
Now our proof of Theorem~\ref{thm;main} is complete.
\end{proof}

\begin{cor}\label{cor;main}
For any subset $\cY\subset\cX$, the following are equivalent.
\begin{enumerate}
\item\label{item1} There exists an mm-space $X$ with $\cY\prec X$.
\item\label{item2} For any $\ep>0$, there exists an mm-space $X$ with $Y\prec_\ep X$ for any $Y\in\cY$.
\item\label{item3} $\cY$ is $\Box$-precompact.
\end{enumerate}
\end{cor}

\begin{proof}
The implication \eqref{item1} $\Rightarrow$ \eqref{item2} is trivial
and we proved the implication \eqref{item3} $\Rightarrow$ \eqref{item1} in Theorem~\ref{thm;main}.
We shall prove the implication \eqref{item2} $\Rightarrow$ \eqref{item3}.

For $\ep > 0$,
we let $X$ be as in \eqref{item2} and take a finite set $\cK$ of Borel sets of $X$ with
\[
\max_{K\in\cK} \diam K \le\ep \quad\text{ and }\quad
\mux(\bigcup\cK) \ge 1-\ep.
\]

By assumption,
any $Y\in\cY$ admits a Borel map $f_Y :X\to Y$ and a Borel set $\tX_Y \subset X$ with
\begin{enumerate}
\item $\mux(\tX_Y) \ge 1-\ep$,
\item $d_Y (f_Y(x), f_Y(x')) \le d_X( x, x') + \ep$ for any $x, x' \in \tX_Y$, and
\item $\proh((f_Y)_* \mux, \muy) \le \ep$.
\end{enumerate}

We note that the set $\cK_Y := \{U_\ep (f_Y (K\cap\tX_Y)) : K\in\cK \}$ satisfies
\begin{align*}
\#\cK_Y \le\#\cK, \qquad
\diam K <4\ep, \qquad
\diam\bigcup\cK_Y <\diam\bigcup\cK +3\ep <\infty,
\end{align*}
and
\begin{align*}
\muy(\bigcup\cK_Y) = \muy(U_\ep (f_Y (U_Y)))
\ge (f_Y)_* \mux (\overline{f_Y (U_Y)}) -\ep
\ge \mux(U_Y) -\ep \ge 1- 3\ep
\end{align*}
for any $Y\in\cY$ and $K\in\cK_Y$, where $U_Y :=\bigcup\cK\cap\tX_Y$.
Therefore Lemma \ref{lem;precpt} implies that $\cY$ is $\Box$-precompact
and this finishes the proof.
\end{proof}

\section{Proof of Proposition~\ref{prop;GH}}\label{sec;GH}

In this section, we make some preparations and prove Proposition~\ref{prop;GH}.
The facts which we need here also are found in \cite{S} and the references therein.

\begin{defn}[e.g.~{\cite[Definition~3.10]{S}}]
For metric spaces $X$ and $Y$ and $\ep>0$,
we call a map $f:X\to Y$ an \emph{$\ep$-isometry} if it satisfies
\begin{enumerate}
\item $|d_Y (f(x), f(x')) - d_X (x, x')| \le\ep$ for any $x, x' \in X$ and
\item $d_Y (y, f(X)) \le\ep$ for any $y\in Y$.
\end{enumerate}
\end{defn}

We know that the Gromov--Hausdorff distance $d_{GH}$ is complete on the set of all (isometry classes of) compact metric spaces,
e.g.~\cite[Lemma~3.9]{S}.
We need the following rather than its precise definition.
\begin{lem}[e.g.~{\cite[Lemma~3.11]{S}}]\label{lem;episomGH}
Let $X$ and $Y$ be compact metric spaces and $\ep>0$.
\begin{enumerate}
\item If $d_{GH} (X, Y)<\ep$, then there exists a $2\ep$-isometry $f:X\to Y$.
\item If there exists an $\ep$-isometry $f:X\to Y$, then $d_{GH} (X, Y)<2\ep$.
\end{enumerate}
\end{lem}

\begin{defn}\label{def;precep}
Let $X$ be a compact metric space and  $\ep>0$.

For a compact metric space $Y$,
we write $Y\prec_\ep X$ if there exists a map $f:X\to Y$ with
\begin{enumerate}
\item $d_Y (f(x), f(x')) \le d_X (x, x') +\ep$ for any $x, x' \in X$ and
\item $d_Y (y, f(X)) \le\ep$ for any $y\in Y$.
\end{enumerate}

For a family $\cY$ of compact metric spaces, we write $\cY\prec_\ep X$ if $Y\prec_\ep X$ for any $Y\in\cY$.
\end{defn}

The following is an analog of Lemma~\ref{lem;key}.
\begin{lem}\label{lem;GHkey}
Let $\cY$ be a GH-precompact family of compact metric spaces and $\ep, \ep'>0$.
If a finite metric space $X$ satisfies $\cY\prec_\ep X$,
then there exists a finite metric space $X'$ with $d_{GH} (X, X')<6\ep$ and $\cY\prec_{\ep'} X'$.
\end{lem}

\begin{proof}
Since there is nothing to prove if $\ep\le\ep'$, we may assume that $\ep' <\ep$.
We suppose that $\cY\prec_\ep X$
and fix a map $f_Y :X\to Y$ as in Definition~\ref{def;precep} and a finite subset $N_Y \subset Y$ with $U_{\ep'}(N_Y) = Y$ for each $Y\in\cY$.
Since $\cY$ is GH-precompact,
there exists a finite subset $\cY' \subset\cY$ with $d_{GH} (Y, \cY')<\ep'$ for any $Y\in\cY$.

For each $x\in X$ and $Y\in\cY'$,
we put
\[
N_Y (x) := \{ y\in N_Y : d(f_Y (x), y) \le\ep \} \ne \emptyset.
\]
We define
\[
X'_x := \prod_{Y\in\cY'} N_Y (x) \text{ for } x\in X, \qquad X' := \bigsqcup_{x\in X} X'_x,
\]
and
\[
d_{X'} (\phi, \phi') := \max\left\{ \max_{Y\in\cY'} d_Y (\phi(Y), \phi'(Y)), \, d_X (x, x') \right\}
\]
for $\phi\in X'_x$ and $\phi' \in X'_{x'}$ with $x, x' \in X$.
Then $(X', d_{X'})$ is a finite metric space.

The natural surjection $\pi:X' \to X$ defined by $\pi(\phi)=x$ if $x\in X$ and $\phi\in X'_x$ satisfy
\[
d_X (x, x') \le d_Y (\phi(Y), \phi'(Y)) \le d_Y (f_Y(x), f_Y(x')) + 2\ep \le d_X (x, x') + 3\ep
\]
and Inequality~\eqref{ineq;3ep}
for any $\phi\in X'_x$ and $\phi' \in X'_{x'}$ with $x, x' \in X$ and $Y \in \cY'$.
This means that $\pi$ is a $3\ep$-isometry and Lemma~\ref{lem;episomGH} yields $d_{GH}(X, X')<6\ep$.

We fix $Y\in\cY'$ and consider the map
\[
g:X' \to Y, \qquad g(\phi) = \phi(Y).
\]
Then we have
\[
d_Y (g(\phi), g(\phi')) = d_Y(\phi(Y), \phi'(Y)) \le d_{X'} (\phi, \phi')
\]
for any $\phi, \phi' \in X'$.
If $y\in Y$, there exist $y' \in N_Y$ with $d_Y (y, y')<\ep'$ and $x\in X$ with $d_Y (y', f_Y (x))\le\ep$.
This means that $y' \in N_Y (x)$ and there exists $\phi\in X'_x \subset X'$ with $g(\phi)=y'$ and hence
\[
d_Y (y, g(X')) \le d_Y (y, y') < \ep'.
\]
These imply that $\cY'\prec_{\ep'} X'$.

Then $\cY\prec_{5\ep'} X'$ by Lemma~\ref{lem;episomGH}.
This completes the proof of Lemma~\ref{lem;GHkey}.
\end{proof}

The following is an analog of Lemma~\ref{lem;final},
cf.~\cite[Lemma~5.33]{S}.
\begin{lem}\label{lem;GHfinal}
Let $\seqn{\ep_n}$ be a sequence of positive numbers with $\ep_n \to 0$ as $\nti$.
If sequences $\seqn{X_n}$ and $\seqn{Y_n}$ of compact metric spaces GH-converge to compact metric spaces $X$ and $Y$ respectively and $Y_n \prec_{\ep_n} X_n$ for any $n$,
then $Y\prec X$.
\end{lem}

\begin{proof}
Let $f_n : X_n \to Y_n$ be as in Definition~\ref{def;precep} with $\ep=\ep_n$.
Lemma~\ref{lem;episomGH} states that there exist $\delta_n$-isometries $\Phi_n : X\to X_n$ and $\Psi_n : Y_n \to Y$ with $\delta_n \to 0$ as $\nti$.
We put $g_n := \Psi_n \circ f_n \circ\Phi_n : X\to Y$ and
take countable dense sets $D\subset X$ and $E\subset Y$ and $x_{y, n} \in X$ with
\[
d_Y (g_n (x_{y, n}), y)\le 2\ep_n +4\delta_n \text{ for each } y\in E \text{ and } n.
\]
By a diagonal argument,
we take a subsequence $\{g_{n(k)}\}_{k=1}^\infty$ of
$\seqn{g_n}$ for which $\{g_{n(k)} (x)\}_{k=1}^n$ converges to some $y_x \in Y$ for any $x\in D$
and $\{x_{y, n(k)}\}_{k=1}^n$ converges to some $x_y \in X$ for any $y\in E$.
Then there exists a unique $1$-Lipschitz map $f:X\to Y$ with $f(x)=y_x$ for any $x\in D$.

If $y\in Y$, $\ep>0$, and $k$ is large enough, there exists $x\in D$ with $d_X (x_y, x)<\ep$ and
\begin{align*}
d_Y (f(x_y), y)
&\le d_Y (f(x_y), f(x)) +d_Y (f(x), g_{n(k)} (x)) \\
&\qquad +d_Y (g_{n(k)} (x), g_{n(k)} (x_{y, n(k)})) + d_Y (g_{n(k)} (x_{y, n(k)}), y) \\
&< 2d_X (x_y, x) + d_X (x_y, x_{y, n(k)}) +\ep <4\ep.
\end{align*}
This implies that $f(x_y)=y$ and $f$ is surjective.
This finishes the proof.
\end{proof}

Now we are ready to prove Proposition~\ref{prop;GH}.
\begin{proof}[Proof of Proposition~\ref{prop;GH}]
We suppose that $\cY$ is a GH-precompact family of compact metric spaces and
fix a sequence $\seqn{\ep_n}$ of positive numbers $\ep_n>0$ with $\sum_\niN \ep_n <\infty$.
Any compact metric space $X_0$ satisfies $\cY\prec_{\ep_0} X_0$ with $\ep_0 := \sup_{Y\in\cY} \diam Y <\infty$.
We use Lemma~\ref{lem;GHkey} to find compact metric spaces $X_n$
with $\cY\prec_{\ep_n} X_n$ and $d_{GH} (X_n, X_{n+1}) <6\ep_n$ for any $\niN$.
Since $\seqn{X_n}$ is Cauchy,
it GH-converges to some compact metric space $X$.
Then $\cY\prec X$ by Lemma~\ref{lem;GHfinal} and this finishes the proof.
\end{proof}

The following is an analog of Corollary~\ref{cor;main}.
\begin{cor}
For any family $\cY$ of compact metric spaces, the following are equivalent.
\begin{enumerate}
\item\label{item;GH1} There exists a compact metric space $X$ with $\cY\prec X$.
\item\label{item;GH2} For any $\ep>0$, there exists a compact metric space $X$ with $\cY\prec_\ep X$.
\item\label{item;GH3} $\cY$ is GH-precompact.
\end{enumerate}
\end{cor}

\begin{proof}
The implications \eqref{item;GH1} $\Rightarrow$ \eqref{item;GH2} is trivial
and we proved the implication \eqref{item;GH3} $\Rightarrow$ \eqref{item;GH1} in Proposition~\ref{prop;GH}.
We shall prove the implication \eqref{item;GH2} $\Rightarrow$ \eqref{item;GH3}.

For $\ep > 0$, we let $X$ be as in \eqref{item;GH2} and take a finite subset $N\subset X$ with
\[
d_X (x, x')>\ep \quad\text{ and }\quad  d_X (x'', N)\le\ep
\]
for any $x, x' \in N$ with $x\ne x'$ and any $x'' \in X$.
By assumption, any $Y\in\cY$ admits a map $f_Y :X\to Y$ as in Definition~\ref{def;precep} and we put $N_Y := f_Y (N) \subset Y$.
Then we have
\[
\diam{Y} \le \diam{X}+3\ep, \quad \# N_Y \le\# N<\infty, \quad\text{ and }\quad d_Y (y, N_Y) \le 2\ep
\]
for any $Y\in\cY$ and $y\in Y$.
Therefore $\cY$ is GH-precompact by e.g.~\cite[Lemma~3.12]{S} and this finishes the proof.
\end{proof}

\begin{rem}
We wonder if there is any relation between Theorem~\ref{thm;main} and Proposition~\ref{prop;GH}.

A compact metric space $\cL_1 (X)$ is associated with an mm-space $X$,
e.g.~\cite[Definition~4.43]{S}.
If $\cY\subset\cX$ is a $\Box$-precompact set,
then $\{\cL_1 (Y) : Y\in\cY\}$ is GH-precompact,
cf.~\cite[Proposition~5.5, Lemma~7.7]{S}.

A pyramid $\cP$ is said to be \emph{concentrated} if $\{\cL_1 (X) : X\in\cP\}$ is GH-precompact, e.g.~\cite[Definition~7.9]{S}.
If $X_n := S^1 \times\dots\times S^n$ is the product of the unit spheres $S^n$ in $\R^{n+1}$,
then $\overline{\bigcup_{n=1}^\infty \cP_{X_n}}$ is a concentrated pyramid which is not $\Box$-precompact,
e.g.~\cite[Corollary~7.10]{S},
cf.~\cite[Remark~13]{OY}.
\end{rem}

\subsection*{Acknowledgements}
The authors thank Professor Takashi Shioya for his comments.
The first author was partly supported by JSPS KAKENHI (No.20J00147)
and the second author was partly supported by JSPS KAKENHI (No.18K03298).

\end{document}